\documentclass[11pt]{article}
\usepackage{amsmath,amsthm}
\usepackage{mathrsfs}
\usepackage{amssymb}
\usepackage{amsfonts}
\usepackage{bm}
\usepackage{mathtools}
\usepackage{tikz}
\usetikzlibrary{calc,shapes, backgrounds}
\pgfdeclarelayer{edgelayer}
\pgfdeclarelayer{nodelayer}
\pgfsetlayers{edgelayer,nodelayer,main}

\usetikzlibrary{calc,shapes, backgrounds, arrows,positioning}
\usepackage{xcolor}
\usepackage{geometry}
\usepackage[pdfstartview=FitH,
bookmarksopen=false,
colorlinks=true,
linkcolor=blue,
citecolor=blue,
urlcolor=blue]{hyperref}

\makeatletter
\def\@seccntDot{.}
\def\@seccntformat#1{\csname the#1\endcsname\@seccntDot\hskip 0.5em}
\renewcommand\section{\@startsection{section}{1}{\z@}%
{18\p@ \@plus 6\p@ \@minus 3\p@}%
{9\p@ \@plus 6\p@ \@minus 3\p@}%
{\large\bfseries\boldmath}}
\renewcommand\subsection{\@startsection{subsection}{2}{\z@}%
{12\p@ \@plus 6\p@ \@minus 3\p@}%
{3\p@ \@plus 6\p@ \@minus 3\p@}%
{\bfseries\boldmath}}
\renewcommand\subsubsection{\@startsection{subsubsection}{3}{\z@}%
{12\p@ \@plus 6\p@ \@minus 3\p@}%
{\p@}%
{\bfseries\boldmath}}
\makeatother

\usepackage{microtype}

\theoremstyle{plain}
\newtheorem{theorem}{Theorem}
\newtheorem{lemma}{Lemma}
\newtheorem{corollary}{Corollary}
\newtheorem{proposition}{Proposition}
\newtheorem{conjecture}{Conjecture}

\newtheorem{definition}{Definition}

\numberwithin{equation}{section}
\allowdisplaybreaks
\begin{document}

\title{The Cop Number of Graphs with Forbidden Induced Subgraphs}
\author{Mingrui Liu
\thanks{Yau Mathematical Sciences Center, Tsinghua University, Beijing 100084, China
({\tt lmr16@mails.tsinghua.edu.cn}). }}

\maketitle
\begin{abstract}
\par\vspace{2mm}

In the game of Cops and Robber, a team of cops attempts to capture a robber on a graph $G$. Initially, all cops occupy some vertices in $G$ and the robber occupies another vertex. In each round, a cop can move to one of its neighbors or stay idle, after which the robber does the same. The robber is caught by a cop if the cop lands on the same vertex which is currently occupied by the robber. The minimum number of cops needed to guarantee capture of a robber on $G$ is called the {\em cop number} of $G$, denoted by $c(G)$.

We say a family $\cal F$ of graphs is {\em  cop-bounded} if there is a constant $M$ so that $c(G)\leq M$ for every graph $G\in \cal F$.
Joret, Kamin\'nski, and Theis [Contrib. Discrete Math. 2010] proved that the class of all graphs not containing a graph $H$ as an induced subgraph is cop-bounded if and only if $H$ is a linear forest;
morerover, $C(G)\leq k-2$ if 
if $G$ is induced-$P_k$-free for $k\geq 3$.
In this paper, we consider the cop number of a family of graphs forbidding certain two graphs and generalized some previous results.

%Sivaraman and Testa recently proved
%$c(G)\leq 2$ if $G$ is induced-$\{2P_2, C_k\}$-free for $k\in \{3,4,5\}$.

%we put forward a new point of view for $H=P_k$ and provide some new results about the forbidden induced subgraph $H$.

\end{abstract}

\section{Introduction}
The game of $\textit{Cops and Robber}$ is played on a connected graph $G$ by two players: $\cal C$ and $\cal R$. The player $\cal C$ controls $k$ pieces (cops) for some integer $k\geq 1$, and the player $\cal R$ has only one piece (the robber). In the beginning of the game, $\cal C$ put down 
$k$ cops on some vertices of $G$ (not necessary distinct). Then $\cal R$ put down the robber on a vertex of $G$. Then, the two players move their pieces alternately. In each of the succeeding rounds, each piece either moves to a vertex adjacent to its current position or stays still. 
%with the robber taking his turn after the cop has moved all of his pieces.
We say the game is {\em cop-win} if after a finite rounds, a cop and the robber meet at the same vertex (that is, the cops catch the robber). It is {\rm robber-win} if the robber can avoid being caught by cops after arbitrarily many rounds. Both players have complete information, that is, they know the graph and the positions of all the pieces. In this paper, we only consider the game played on finite undirected graphs. 

The most important problem we considered in this game is how many cops we need to capture the robber on a given graph. The minimum number of cops guaranteed to capture the robber on the graph $G$ in finite rounds is called the $\textit{cop number}$ of $G$ and denoted by $c(G)$. Nowakowski and Winkler \cite{MR685631} and Quilliot \cite{name1} characterized the class of graphs with cop number $1$. Aigner and Fromme \cite{MR739593}
proved any planar graph $G$ has cop number at most $3$.
Finding a combinatorial characterization of graphs with cop number $k$ (for $k\geq 2$) is a major open problem in this field. Moreover, determining the cop number for a given graph is also a computationally hard problem \cite{MR2757374}. We even don't know the order of the cop number upon the number of vertices of any graph $G$ \cite{MR890640}. For more information about the game of cops and robber we recommend the book by Bonato and Nowakowski \cite{MR2830217} and the survey by Baird and Bonato \cite{MR2980752}.

\begin{definition}
Let $G_1,G_2,\ldots, G_n$ be some graphs, we say a graph $G$ is \{$G_1$,$G_2$,\ldots, $G_n$\}-free if $G$ does not contain any $G_i$ as an induced subgraph. We say $G$ is ($G_1+G_2+\ldots+G_n$)-free if $G$ does not contain the disjoint union of all the graphs $G_1,G_2,\ldots, G_n$ as an induced subgraph.
\end{definition}

\begin{definition}
For a vertex $v$ in a graph $G$, the closed neighbourhood of $v$, denoted by $N[v]$ is the set of all neighbors of $v$, including $v$ itself. The open neighbourhood of $v$, denoted by $N(v)$ is the set of all neighbors of $v$. What's more, $N[v_1,v_2]$ is the union of $N[v_1]$ and $N[v_2]$. 
\end{definition}

\begin{definition}
The graph consisting of a path on $k$ vertices is denoted by $P_k$.
\end{definition}

\begin{definition}
For the graph $G$, the subgraph $G[U]$ is the induced subgraph generated by the subset $U$ of the vertices of $G$.
\end{definition}

Joret, Kaminski and Theis \cite{MR2791289} proved that $k-2$ cops can catch the robber in any $P_k$-free graph for every interger $k\geq 3$. In this paper, we prove this theorem with a new method and use this idea to explore the cop number for a graph $G$ with some forbidden induced structures.

\begin{theorem}\label{thm:Pk}\cite{MR2791289}
Let $G$ be a $P_k$-free graph with $k \geq 3$, then $k-2$ cops can catch the robber in $G$. That is, $c(G)\leq k-2$.
\end{theorem}

The idea we use to prove the above theorem can also be used to prove the following two theorems and we can bound the cop number of graphs with some forbidden induced subgraphs with the cop number.

\begin{theorem}\label{thm:two}
If we forbid the following two structures in any graph $G$, then $c(G)\leq l+1$ for $l \geq 1$.
\begin{figure}[hbt]
\begin{center}
\tikzstyle{vertex}=[circle,fill=black,inner sep=1.5pt]
\tikzstyle{new}=[circle,fill=red,inner sep=1.5pt]
\tikzstyle{blue}=[circle,fill=blue, inner sep=1.5pt]
\tikzstyle{green}=[circle,fill=green, inner sep=1.5pt]

\tikzstyle{curly edge}=[]
\tikzstyle{straight edge}=[]

\tikzstyle{straight edge}=[]
\tikzstyle{vertex}=[circle,fill=black,inner sep=1.5pt]

\begin{tikzpicture}[scale=0.8]
	\begin{pgfonlayer}{nodelayer}
		\node [style=vertex] (0) at (0, 0.5) {};
		\node [style=vertex] (1) at (0, -0.5) {};
		\node [style=vertex] (2) at (1, 0) {};
		\node [style=vertex] (3) at (2, 0) {};
		\node [style=vertex] (4) at (3, 0) {};
	\end{pgfonlayer}
	\begin{pgfonlayer}{edgelayer}
		\draw [style=straight edge] (1) to (2);
		\draw [style=straight edge] (0) to (2);
		\draw [style=straight edge] (2) to (3);
		\draw [style=dashed] (3) to (4);
	\end{pgfonlayer}
	\node at (2,-.6) {$\underbrace{\hspace{4 em}}_{l+1 \hbox{ vertices}}$};
\end{tikzpicture}
\hspace{1cm}% NO SPACE!
\begin{tikzpicture}[scale=0.8]
	\begin{pgfonlayer}{nodelayer}
		\node [style=vertex] (0) at (0, 0.5) {};
		\node [style=vertex] (1) at (0, -0.5) {};
		\node [style=vertex] (2) at (1, 0) {};
		\node [style=vertex] (3) at (2, 0) {};
		\node [style=vertex] (4) at (3, 0) {};
	\end{pgfonlayer}
	\begin{pgfonlayer}{edgelayer}
		\draw [style=straight edge] (0) to (1);
		\draw [style=straight edge] (1) to (2);
		\draw [style=straight edge] (0) to (2);
		\draw [style=straight edge] (2) to (3);
		\draw [style=dashed] (3) to (4);
	\end{pgfonlayer}
	\node at (2,-.6) {$\underbrace{\hspace{4 em}}_{l+1 \hbox{ vertices}}$};
\end{tikzpicture}
\hspace{1cm}% NO SPACE! 
\label{fig:graphs}
\end{center}
\caption{The structures mentioned in Theorem 2}
\end{figure}
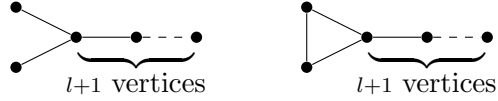
\end{theorem}

\begin{theorem}\label{thm:claw}
Let $G$ be a $\{P_k,K_{1,3}\}$-free graph for $k \geq 5$, then $c(G) \leq k-3$.
\end{theorem}

\noindent\textbf{Remark:} Some mathematicians refer to $K_{1,3}$ as a "claw", we shall use $K_{1,3}$ and "claw" interchangeably.

Next, we will explore the graph $G$ with forbidden induced subgraphs consisted of different paths, we find that we can use fewer cops to catch the robber than we forbid the disjoint union of the same smaller paths.

\begin{theorem}\label{thm:PP}
Let $G$ be a $(P_{i_1}+P_{i_2}+\ldots+P_{i_k})$-free graph for $k\geq 2$, $i_j\geq 1$ and one of $i_j\geq 3$ ($1\leq j \leq k$). Then $c(G)\leq i_1+i_2+\ldots +i_k-2$.
\end{theorem}

\section{Forbidden Induced Structure}
Now we begin to prove \autoref{thm:Pk}, the most important thing to prove this theorem is that similar idea can be used in the proof of other theorems in our paper.
\begin{proof}[Proof of \autoref{thm:Pk}]
In the game of cops and robber, we can define a function $f$, $f$ means the minimal distance between the cop pieces and the robber after each round. Let $t$ be the minimum value of $f$ over all rounds. If $t=0$, it means that at least one cop and the robber are on the same vertex after finite rounds. On the other hand, if $t>0$, it means that all the cop pieces and the robber can never be on the same vertex in the whole game.

Suppose that the cop player can never use $k-2$ cop pieces to capture the robber in the graph $G$. That means $t \geq 1$ in the whole game. We can think in another way that after finite rounds, there exists one cop $c_1$ on the vertex $u$ whose distance is $t$ ($t\geq 1)$ from the robber $r$ on the vertex $v$. We can find a shortest path $P_{t+1}$: $u-u_1-\ldots-u_{t-1}-v$ between $u$ and $v$. Since there is no difference between each piece of the cop player, we can suppose all the cop pieces: $c_1,c_2,\ldots, c_{k-2}$ are on the vertex $u$ without loss of generality.  

Now it's the robber's turn to move. In order to keep the minimal distance not decreasing, the robber $r$ must move to a new vertex $v_1$ and the vertex $v_1$ is not adjacent to the vertices $u,u_1,\ldots,u_{t-1}$. That means we can find a shortest path $P_{t+2}$: $u-u_1-\ldots-v-v_1$. Then the cop player will move $k-3$ pieces $c_2,\ldots,c_{k-2}$ to the vertex $u_1$. Since up to now the minimal distance between the cops and the robber is also $t$, the robber must move to another new vertex $v_2$, $v_2$ is not adjacent to the vertices $u,u_1, \ldots, u_{t-1},v$ for the same reason of the first round and we can get a shortest path $P_{t+3}$:$u-u_1-\ldots-v-v_1-v_2$. The cop player will move $k-4$ pieces $c_3,\ldots,c_{k-2}$ to the vertex $u_2$. We use this strategy repeatedly until we don't have cop piece to move to the next vertex, and right now we get a shortest path $P_{t+k-1}$: $u-u_1-\ldots-u_{t-1}-v-v_1-\ldots-v_{k-2}$ and the first $k-2$ vertices are each occupied by one cop.
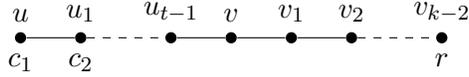
\begin{figure}[hbt]
\begin{center}

\tikzstyle{straight edge}=[]
\tikzstyle{vertex}=[circle,fill=black,inner sep=1.5pt]

\begin{tikzpicture}[scale=0.8]
	\begin{pgfonlayer}{nodelayer}
		\node [style=vertex, label=above:{$u$},label=below:{$c_1$}] (0) at (0,0) {};
		\node [style=vertex, label=above:{$u_1$},label=below:{$c_2$}] (1) at (1,0) {};
		\node [style=vertex, label=above:{$u_{t-1}$}]  (2) at (2.5,0) {};
		\node [style=vertex, label=above:{$v$}] (3) at (3.5,0) {};
		\node [style=vertex, label=above:{$v_1$}] (4) at (4.5,0) {};
		\node [style=vertex, label=above:{$v_2$}] (5) at (5.5,0) {};
		\node [style=vertex, label=above:{$v_{k-2}$},label=below:{$r$}] (6) at (7,0) {};
	\end{pgfonlayer}
	\begin{pgfonlayer}{edgelayer}
		\draw [style=straight edge] (0) to (1);
		\draw [style=dashed] (1) to (2);
		\draw [style=straight edge] (2) to (3);
		\draw [style=straight edge] (3) to (4);
		\draw [style=straight edge] (4) to (5);
		\draw [style=dashed] (5) to (6);
	\end{pgfonlayer}
\end{tikzpicture}
\end{center}
\label{fig:extendedblockgraph}
\end{figure}

Based on our assumption, $t$ is greater than or equal to $1$, so $t+k-1\geq k$. That means we can get a shortest path $P_k$, however the graph $G$ is $P_k$-free. We draw a contradiction, so we can use $k-2$ cops to catch the robber in the graph $G$.
\end{proof}

The idea we use to prove \autoref{thm:Pk} is very useful to determine the cop number for some graphs with forbidden induced structure. It means that, if we can't use $l$ cops to catch the robber in any graph $G$, we can find an induced path $P_{l+t+1}$: $u_0-\ldots-u_{l+t}$ for $t\geq 1$ in the graph $G$. We can also guarantee the first vertices: $u_0,\ldots, u_{l-1}$ of this path are occupied by the cop $c$ and the last vertex $u_{l+t}$ is occupied by the robber $r$. Next we shall use this idea to prove the following theorems.

\begin{proof}[Proof of \autoref{thm:two}] Suppose not, we can use $l$ cops to form an induced $P_{l+2}$ path: $v_0-v_1-\ldots-v_{l+1}$ and each vertex $v_i$ is occupied by one cop $c_{i+1}$ for $0\leq i \leq l-1$. The robber is on the vertex $v_{l+1}$ by the details in \autoref{thm:Pk}. Now it's the turn for the cop player to move. We let each piece of the cop player move forward to the next vertex along this path. Then the robber cannot stay still, he must move. If he has more than two choices, for example $v_{l+2}, v_{l+2}', \ldots$. Whether or not these vertices are adjacent to each other, the vertices $v_1,\ldots, v_{l+1}$ already form a shortest path $P_{l+1}$. We need to forbid these two structures mentioned in Theorem 2, so one of any two these vertices must be adjacent to one of vertices $v_1,\ldots, v_l$. If either vertex $v_{l+2}$, $v_{l+2}'$, say $v_{l+2}$ is adjacent to any of vertices $v_1, \ldots, v_l$, the robber cannot move to $v_{l+2}$ since $v_1, \ldots, v_l$ are occupied by the cops. There is only one vertex except $v_l$ in the neighbourhood of $v_{l+1}$ which is not adjacent to the vertices $v_1,\ldots, v_l$. That means the robber has only one vertex left to move, we call it $v_{l+2}$. Then the cops will still move forward and the robber will meet the same problem as in the former round. So the robber still has only one choice to move to escape from the capture. However, the graph $G$ is finite and the robber's available vertices will form an induced path or cycle. If it's an induced path, we can use these $l$ cops to catch the robber directly. If it's an induced cycle, we use these $l$ cops to keep the robber moving along this cycle and another cop $c_{l+1}$ to move to any vertex in this cycle and wait for the robber. 
So we can use $l+1$ cops to catch the robber in the graph $G$.

\end{proof}

The same idea in \autoref{thm:Pk} and \autoref{thm:two} is that we always use some cops to pursue the robber and we find that we can get an induced $P_k$ in the graph $G$ by analyzing the process of chase. In \autoref{thm:two}, we use one more cop to wait for the robber since there may exist an induced cycle by analyzing the trail of the robber's movements. Now we can use what we already know to prove the \autoref{thm:claw}.

\begin{proof}[Proof of \autoref{thm:claw}]
Now, we have $k-3$ cop pieces and we can use them to form an induced path $P_{k-1}$: $v_0-v_1-\ldots-v_{k-2}$, the first $k-3$ vertices are occupied by the cops $c_1,\ldots,c_{k-3}$, the robber $r$ is on the vertex $v_{k-2}$. It's the turn for the cop player to move and all cop pieces will move forward along this path, the vertices $v_1,\ldots,v_{k-3}$ are occupied by the cops as a result. Right now the robber is adjacent to the cop $c_{k-3}$ which is on the vertex $v_{k-3}$. All the existing vertices are occupied by the cops or adjacent to one vertex occupied by the cop, so the robber must choose a new vertex $v_{k-1}$ to escape from the capture. However, the path $P_k$ is forbidden in the graph $G$, the new vertex $v_{k-1}$ must be adjacent to at least one of vertices: $v_0,v_1,\ldots,v_{k-3}$. On the another hand, the robber cannot be adjacent to the vertex occupied by the cop directly. The only choice for the vertex $v_{k-1}$ is to be adjacent to the vertex $v_0$ and we have a $C_k$ cycle. The cop player moves the cop $c_1$ backward to the vertex $v_0$ and the other $k-4$ cops move forward along this path. The robber $r$ is adjacent to the vertices $v_0$ and $v_{k-2}$ occupied by the cops $c_1$ and $c_{k-3}$ respectively, so the robber must move to a new vertex $v_k$. However, the vertices \{$v_{k-2},v_{k-1},v_k,v_0$\} will form an induced "claw". In order not to let this happen, the vertex $v_k$ must be adjacent to the vertex $v_0$ or $v_{k-2}$ since $v_0$ and $v_{k-2}$ are not adjacent ($v_0$ and $v_{k-2}$ are vertices of induced path $P_{k-1}$). However, the cop player has already prepared the cops for the vertices $v_0$ and $v_{k-2}$, so the robber will not escape from capture by the cops. That means we can use $k-3$ cops to catch the robber in $G$.

\end{proof}

If we carefully analyze the proof of the \autoref{thm:claw}, we can find that if we use one fewer cop to catch the robber in the $P_k$-free graphs, it must form an induced cycle $C_k$. Also, if we just keep some cops still and make the other cops move forward in the proof of \autoref{thm:claw}, we can also get an induced cycle $C_j$ for $3\leq j\leq k$. So we can get the following Corollary.

\begin{corollary}\label{cor:cy}
Let $G$ be a $\{P_k,C_l\}$-free graph for $k \ge 5$ and $3 \leq l \leq k$, then $c(G) \leq k-3$.
\end{corollary}

Right now, we can answer the following conjectures made by Sivaraman and Testa \cite{1903.11484} partially.

\begin{conjecture}
Let $G$ be a $P_5$-free graph, then $c(G)\leq 2$.
\end{conjecture}
\begin{conjecture}
Let $G$ be a $2P_2$-free graph, then $c(G)\leq 2$.
\end{conjecture}

$2P_2$-free means ($P_2+P_2$)-free, we shall use the lemma made by Sivaraman and Testa \cite{1903.11484}.

\begin{lemma}\label{lm:3} \cite{1903.11484}
Let $G$ be a $2P_2$-free graph with diameter $3$. Then $c(G)\leq 2$.
\end{lemma}

\begin{theorem}
Let $G$ be a \{$2P_2$, $\overline P_5$\}-free graph, then $c(G)=2$.
\end{theorem}
\begin{proof}
We already know that if $G$ is $2P_2$-free and the diameter of $G$ is 3, we can use 2 cops to capture the robber by Lemma 1. So we just need to consider the case where the diameter of $G$ is $2$. The graph of $\overline P_5$ (complement graph of $P_5$) is like a house and we call it "house" in our proof. The graph $G$ must contain an induced $P_4$ since we already know that we can use two cops to catch the robber in any $P_4$-free graph by \autoref{thm:Pk}. We put two cops on the two internal vertices $u_0$ and $v_0$ of this $P_4$: $u_1-u_0-v_0-v_1$ respectively. The robber $r$ will choose one vertex $z_0$ which is not adjacent to $u_0$ or $v_0$ to stay. Then the vertices except these two vertices in the graph $G$ are divided into four parts: the cop $c_1$'s private neighbourhood (the vertex adjacent to $c_1$ not adjacent to $c_2$) denoted by $U$, the cop $c_2$'s private neighbourhood (the vertex adjacent to $c_2$ not adjacent to $c_1$) denoted by $V$, the common neighbourhood of $c_1$,$c_2$ (the vertex adjacent to both $c_1$ and $c_2$) denoted by $W$ and the vertices not adjacent to $c_1$ or $c_2$ denoted by $Z$. The vertices in $Z$ is an independent set since all the vertices in $Z$ are not adjacent to $P_2$: $u_0-v_0$. 
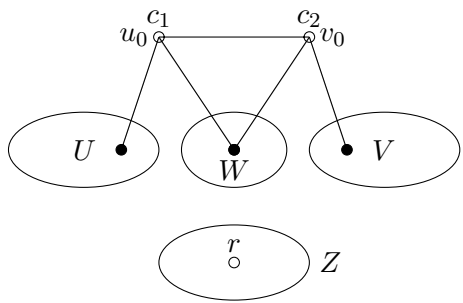
\begin{center}
\begin{tikzpicture}
	\draw  (1,0) circle (2pt) -- (3,0) circle  (2pt);
\filldraw (1,0) -- (0.5,-1.5) circle  (2pt);
\filldraw (1,0) -- (2,-1.5) circle  (2pt);
\filldraw (3,0) -- (2,-1.5) circle  (2pt);
\filldraw (3,0) -- (3.5,-1.5) circle  (2pt);
\draw (1,-1.5) arc (0:360:1cm and 0.5cm) ;
\draw (2.7,-1.5) arc (0:360:0.7cm and 0.5cm) ;
\draw (5,-1.5) arc(0:360:1cm and 0.5cm);
\draw (3,-3) arc(0:360:1cm and 0.5cm);
\draw (2,-3) circle (2pt) node[above]{$r$};
\draw (3,-3) node[right]  {$Z$};
\draw (0.3,-1.5) node[left] {$U$};
\draw (2,-1.5) node[below] {$W$};
\draw (3.7,-1.5) node[right] {$V$};
\draw (1,0) node[above]{$c_1$};
\draw (1,0) node[left]{$u_0$};
\draw (3,0) node[above]{$c_2$};
\draw (3,0) node[right]{$v_0$};

\end{tikzpicture}
\end{center}
\noindent Case 1: If $z_0$ is adjacent to some vertex $w_0$ in $W$, which means the robber is adjacent to one common neighbor of $c_1$ and $c_2$. 

We move the cop $c_1$ to the vertex $w_0$, the robber must move and he has three possible choices to move: the sets $U$, $W$, $V$. The parts $W$ and $V$ are controlled by the cop $c_2$ so the robber cannot move to these parts actually. The available choice for the robber is to move to the part $U$ to escape from the capture. For example, the robber moves to one vertex $u'$ in $U$, this vertex must be adjacent to the vertex $w_0$ to prevent from forming an induced "house":\{$v_0$, $u_0$, $w_0$, $u'$, $z_0$\}. So the cop $c_1$ will catch the robber in the next round.\\

\noindent Case 2: If $z_0$ is not adjacent to any vertex in $W$ or $W$ is an empty set.

Since the diameter of $G$ is 2, the vertex $z_0$ must be adjacent to some vertex in $U$ and $V$. These two vertices in $U$ and $V$ are denoted by $u_1$ and $v_1$ respectively for convenience. We move the cop $c_1$ to the vertex $u_1$ and the cop $c_2$ to the vertex $v_1$. The robber must move to escape from the capture. The only choices available for him are the parts $U$ and $V$. Without loss of generality, let the robber move to one vertex $u_2$ in $U$ which is not adjacent to $u_1$. In the next round, let the cop $c_1$ move back to the vertex $u_0$ and the cop $c_2$ stay still. The only available choices for the robber are the parts $Z$ and $V$.

Case 2(i): If the robber moves to some vertex $z_1$ in $Z$. This vertex must be different from the vertex $z_0$ since the cop $c_2$ is still controlling the vertex $z_0$. Actually, the vertex $z_1$ doesn't exist since the vertex \{$z_1-u_2$, $v_0-v_1$\} will form an induced $2P_2$. 

Case 2(ii): If the robber moves to some vertex $v_2$ which is not adjacent to the vertex $v_1$ in $V$. Let the cop $c_2$ move back to the vertex $v_0$ and the cop $c_1$ stay still. Then the robber will move back to some vertex $z_2$ in the part $Z$. The vertex $z_2$ must be different from the vertex $z_0$. Suppose not, we can find an induced "house": \{$z_0$, $u_2$, $v_2$, $u_0$, $v_0$\}. The vertex $v_2$ is not adjacent to $z_0$ for the same reason. The vertex $z_2$ must be adjacent to the vertex $v_1$ prevent from forming induced $2K_2$: \{$z_0$-$v_1$, $z_2$-$v_2$\}. The vertex $z_2$ must be adjacent to the vertex $u_2$ to prevent from forming induced $2P_2$: \{$z_2$-$v_1$, $u_0$-$u_2$\}. However, in that way we can get an induced "house": \{$z_2$, $u_2$, $v_2$, $u_0$, $v_0$\}. That means the vertex $z_2$ available for the robber to move actually doesn't exist.

So we can use two cops to catch the robber for each case in \{$2P_2$, $\overline P_5$\}-free graph.

\end{proof}

\noindent\textbf{Remark:} We know that we can use two cops to catch the robber for any \{$2P_2$, $C_i$ ($i=3,4,5$)\}-free graphs, the complement graph $\overline{P_5}$ of $P_5$ contains an induced triangle and $C_4$. So this theorem is more generalized than the theorem for \{$2P_2$, $C_i$ ($i=3,4,5$)\}-free.

Now there is only one theorem left for us, we wouldn't like to show the proof of \autoref{thm:PP} since there are complicated notations in this proof. We only prove two propositions just using the same method in the proof of \autoref{thm:PP}.

\begin{proposition}
Let $G$ be a $(P_1+P_k)$-free graph for $k\geq 3$. Then $c(G)\leq k-1$.
\end{proposition}

\begin{proof}
In the beginning, we put all the $k-1$ cops on any fixed vertex $v_0$. Then the induced subgraph generated by the vertices in $V(G)$ except the vertex $v_0$ and its neighbourhood $N(v_0)$ will form a few components $U_1,U_2,\ldots$. Then robber will choose some vertex $v_1$ in one of these components, we call it $U_1$ since the vertex $v_0$ is occupied by the cops. And also, we know that in the induced subgraph $G[U_1]$ generated by the vertices in the component $U_1$ is $P_k$ forbidden for the reason that all the vertices in $U_1$ is not adjacent to the vertex $v_0$. 

We put one cop remaining on the vertex $v_0$ and all the other $k-2$ cops move to the component $U_1$. If the robber wants to move to the other component, he must go through one of the vertices in $N(v_0)$ or $v_0$, that is impossible since the vertex $v_0$ is always occupied by the cop. So the robber will always stay in the component $U_1$, and the induced subgraph $G[U_1]$ is $P_k$-free, we have $k-2$ cops in the component $U_1$ after the cops finish their movements. However, we already knew that $k-2$ cops can catch the robber in any $P_k$-free graphs for $k\geq3$. As a result, we can use $k-1$ cops to catch the robber in any $(P_k,P_1)$-free graphs for $k\geq3$. 
\end{proof}

\begin{proposition}
Let $G$ be a $(P_2+P_k)$-free graph for $k\geq3$. Then $c(G)\leq k$.
\end{proposition}
\begin{proof}
In the beginning, we can find two adjacent vertices $v_0,v_1$ and we put one cop on the vertex $v_1$ and other $k-1$ cops on the vertex $v_0$. We also find that all the components for the induced subgraph $G[V(G)\backslash N[v_0,v_1]]$ is $P_k$-free and we keep two cops remaining on the vertex $v_0$ and $v_1$, the other $k-2$ cops will catch the robber eventually.
\end{proof}

\section{Concluding Remarks}
In this paper we find a powerful idea to solve the problem about how many cops we need to capture the robber in some graphs with forbidden induced subgraphs. We also answer the conjecture made by Sivaraman and Testa \cite{1903.11484} partially. The difficulty for proving the cop number of $2P_2$-free graphs is $2$ is that Theorem 1 is right for the condition $k\geq 3$. We cannot use "zero" cops to capture the robber in any $P_2$-free graph.

\bibliographystyle{abbrv}

\end{document}